\documentclass{article}
\usepackage{amsmath}
\usepackage{amssymb}
\usepackage{amsthm}
\usepackage{enumerate}
\usepackage{color}
\usepackage[english]{babel}
\usepackage[latin1]{inputenc} 

\newtheorem{theorem}{Theorem}[section]
\newtheorem{corollary}[theorem]{Corollary}
\newtheorem{lemma}[theorem]{Lemma}
\newtheorem{proposition}[theorem]{Proposition}
\newtheorem{remark}[theorem]{Remark}
\newtheorem{example}[theorem]{Example}
\newtheorem{definition}[theorem]{Definition}

\newcommand{\C}{{\ensuremath{\mathbb{C}}}}
\newcommand{\Cm}{{\ensuremath{\C^{p\times n}}}}
\newcommand{\Cn}{{\ensuremath{\C^{n\times p}}}}
\newcommand{\Cnn}{{\ensuremath{\C^{n\times n}}}}
\newcommand{\Cmm}{{\ensuremath{\C^{p\times p}}}}
\newcommand{\Ra}{{\ensuremath{\cal R}}}
\newcommand{\Nu}{{\ensuremath{\cal N}}}
\newcommand{\ra}{{\ensuremath{\text{\rm rk}}}}
\newcommand{\ind}{{\ensuremath{\text{\rm Ind}}}}

\newcommand{\odagger}{\mathrel{\text{\textcircled{$\dagger$}}}}
\newcommand{\core}{\mathrel{\text{\textcircled{$\#$}}}}
\newcommand{\weak}{\mathrel{\text{\textcircled{w}}}}

\usepackage[left=3cm,top=3.5cm,right=3cm,bottom=3.5cm]{geometry}

\begin{document}

\title{The $W$-weighted $m$-weak core inverse}

\author{D.E. Ferreyra\thanks{Universidad Nacional de R\'io Cuarto, CONICET, FCEFQyN, RN 36 KM 601, 5800 R\'io Cuarto, C\'ordoba, Argentina. E-mail: \texttt{deferreyra@exa.unrc.edu.ar}},\,
D. Mosi\' c\thanks{University of Ni\v s, Faculty of Sciences and Mathematics, Vi\v segradska 33, 18000 Ni\v s, Serbia. E-mail: \texttt{dijana@pmf.ni.ac.rs}}
}

\date{}
\maketitle

\begin{abstract}
Recently, Malik and Ferreyra introduced the $m$-weak core inverse for complex square matrices which generalizes the core-EP inverse, the
WC inverse, and therefore the core inverse. The main aim of this paper is to extend the concept of $m$-weak core inverse for complex rectangular matrices. This extension is called the $W$-weighted $m$-weak core inverse. We analyse its existence and uniqueness as solution of a system of matrix equations. We present various properties, representations and  characterizations  of the $W$-weighted $m$-weak core inverse, as well as its applications in solving certain matrix systems. A canonical form  of the $W$-weighted $m$-weak core inverse  is  also provided  by using a simultaneous unitary block upper triangularization of a pair of rectangular matrices.
\end{abstract}



\section{Introduction and preliminaries}

Consider $\Cm$  as the set of all $p\times n$ complex matrices. Let $A\in \Cm$. The conjugate transpose, rank, null space and column space of $A$ are represented by $A^*$, $\ra(A)$, $\Nu(A)$ and $\Ra(A)$, respectively.  The index of $A\in \Cnn$, denoted by $\ind(A)$, is the smallest nonnegative integer $k$ such that $\ra(A^k) = \ra(A^{k+1})$. Moreover, $A^0=I_n$ will refer to the $n \times n$ identity matrix. Standard notations $P_S$ and $P_{S,T}$ denote, respectively, the orthogonal projector onto a subspace $S$ and a projector onto $S$ along $T$ when $\C^n$ is equal to the direct sum of subspaces $S$ and $T$.

The Moore-Penrose inverse of a matrix $A\in \Cm$ is the unique matrix $X=A^\dag \in \Cn$ that satisfies the Penrose equations
\[AXA=A, \quad XAX=X,
\quad (AX)^*=AX, \quad(XA)^*=XA.\]
Denote by $P_A:=AA^{\dag}$ and $Q_A:=A^{\dag}A$ the orthogonal projectors onto $\Ra(A)$ and $\Ra(A^*)$, respectively.

The Drazin inverse of a matrix $A\in \Cnn$ is the unique matrix $X=A^d \in \Cnn$ that satisfies
\[X A^{k+1} =A^k, \quad XAX=X, \quad AX=XA, ~\text{where} ~k=\ind(A).\]
When $\ind(A)=1$,  the Drazin inverse is called the group inverse of $A$ and is denoted by $A^\#$.

The core-EP inverse of a matrix $A\in \Cnn$ of index $k$ is the unique matrix  $A^{\odagger}=X\in \Cnn$ satisfying the conditions
$XAX=X$ and $\Ra(X)=\Ra(X^*)=\Ra(A^k)$ \cite{PrMo}.  When $\ind(A)=1$,  $A^{\odagger}:=A^{\core}=A^{\#}AA^\dag$ becomes the core inverse of $A$ \cite{BaTr1}.

By using the core-EP inverse, Zhou et al. \cite{ZhChZh} defined the $m$-weak group inverse of an element in a ring with involution. Let us recall its definition for the matrix case. If $m\in \mathbb{N}$, the $m$-weak group inverse of $A\in \Cnn$ is the unique matrix $A^{\weak_m}=X\in \Cnn$ that satisfies
\begin{equation*}\label{def m-weak group}
AX^2=X ,\quad AX=(A^{\odagger})^m A^m.
\end{equation*}
When $m=1$, the $m$-weak group inverse is reduced to the WG inverse studied by Wang and Chen in \cite{WaCh}.
Based in this new class of generalized inverses, Ferreyra and Malik \cite{FeMa5} proposed the $m$-weak core inverse. Exactly, if $m\in \mathbb{N}$, the $m$-weak core inverse of $A\in \Cnn$ with $k=\ind(A)$, is the unique matrix $A^{\core_m}=X\in \Cnn$ that satisfies
\begin{equation}\label{def m-weak core}
AX=(A^{\odagger})^m A^m P_{A^m},\quad \Ra(X)\subseteq \Ra(A^k).
\end{equation}
It was proved the the unique solution of \eqref{def m-weak core} is the matrix $A^{\core_m}=A^{\weak_m}P_{A^m}$. When $m=1$, the $m$-weak core inverse reduces to the WC inverse studied in \cite{FeLePrTh}, which is an alternative extension of the core inverse.  Moreover, if $m\ge k=\ind(A)$, it is reduced to the core-EP inverse of $A$.

In 1980, Cline and Greville \cite{ClGr} extended the Drazin inverse to rectangular matrices and it was called the $W$-weighted Drazin inverse.
Let $W\in \Cn$ be a fixed non null matrix. The $W$-weighted Drazin inverse of $A\in \Cm$, denoted by  $A^{d,W}$, is the unique matrix
$X\in \Cm$ satisfying the three equations
\[XWAWX=X, \quad AWX=XWA, \quad   XW(AW)^{k+1}=(AW)^k, ~\text{with}~ k=\max\{\ind(AW),\ind(WA)\}.\]
When $p=n$ and $W=I_n$, we recover the Drazin inverse, that is,  $A^{d,I_n}=A^d$.

The $W$-weighted Drazin inverse satisfies some interesting  properties in term of the Drazin inverse of the square matrices $AW$ and $WA$:
\begin{equation}\label{properties w drazin}
A^{d,W}=A[(WA)^d]^2=[(AW)^d]^2A, \quad A^{d,W}W=(AW)^d,  \quad WA^{d,W}=(WA)^d.
\end{equation}

Similarly, the core-EP inverse was extended to rectangular matrices in \cite{FeLeTh1}. The authors defined the $W$-weighted core-EP inverse of $A$, denoted by $A^{\odagger,W}$, as the unique solution $X=(WAWP_{(AW)^k})^\dag$ of the system
\begin{equation}\label{w cep system}
WAWX=P_{(WA)^k}, \quad \Ra(X)\subseteq \Ra((AW)^k),  ~\text{with}~ k=\max\{\ind(AW),\ind(WA)\}.
\end{equation}
Clearly, if $m=n$ and $W=I_n$, we recover the core-EP inverse, that is,  $A^{\odagger,I_n}=A^{\odagger}$.

The $W$-weighted core-EP inverse can be expressed in term of the $W$-weighted Drazin inverse  \cite{FeLeTh1, GaChPa}
\begin{equation}\label{properties weighted CEP}
A^{\odagger,W}=A[(WA)^{\odagger}]^2=A^{d,W}P_{(WA)^k},
\end{equation}

Recently, by using the $W$-weighted core-EP inverse,  the $W$-weighted $m$-weak group inverse of $A$ was defined in \cite{MoStKa} as the unique matrix $A^{\weak_m,W}=X\in \Cm$ satisfying the following equations
\begin{equation}\label{system weighted m-weak group}
AWX=(A^{\odagger,W}W)^m(AW)^{m-1}A, \quad  AWXWX=X,  ~\text{with}~ k=\max\{\ind(AW),\ind(WA)\}.
\end{equation}
It was proved that the unique solution to \eqref{system weighted m-weak group}  is given by
\begin{equation}\label{def weighted m-weak group}
A^{\weak_m,W}=(A^{\odagger,W}W)^{m+1}(AW)^{m-1}A.
\end{equation}
Note that $A^{\weak_1,W}=A^{\weak,W}$ coincides with the $W$-weighted WG inverse established in \cite{FeOrTh1} while $A^{\weak_m,W}=A^{d,W}$ if $m\ge k$.
When $p=n$ and $W=I_n$, we recover the $m$-weak group inverse, that is,  $A^{\weak_m,I_n}=A^{\weak_m}$.

Other generalized inverses of complex square matrices such as the BT, DMP and CMP inverses were also extended to the rectangular case \cite{FeThTo, Meng, Mo1, MoKo}.

In this paper, we extend the definition given in \eqref{def m-weak core} for the case of a matrix $A\in \Cm$ by  considering a conformable weight $0\neq W\in \Cn$, which gives rise to a new type of generalized inverse called $W$-weighted $m$-weak core inverse of $A$. The main contributions of this paper are structured as follows. In Section 2, we define the $W$-weighted $m$-weak core inverse  and study its main representations and properties.
In Section 3, we provide several algebraic characterizations of the  $W$-weighted $m$-weak core inverse by mean of different systems of matrix equations.
In Section 4, we obtain a canonical form of the $W$-weighted $m$-weak core inverse  by using a simultaneous decomposition of the matrices $A$ and $W$ known as the weighted core-EP decomposition. In consequence, some more properties that this new generalized inverse possesses are investigated.
Finally, in Section 5, we show that the $W$-weighted $m$-weak core inverse can be used to solve certain linear matrix equations and express their general solution.

\section{Existence, Uniqueness and Representations}

In this section, we define and investigate the $W$-weighted $m$-weak core  inverse for rectangular matrices $A\in \Cm$ by considering a fixed non null matrix $W\in \Cn$.

\begin{definition}
Let $A \in \Cm$, $0\neq W\in \Cn$, $k=\max\{\ind(AW), \ind(WA)\}$, and $m\in \mathbb{N}$. The matrix $A^{\core_m,W}\in \Cm$ satisfying
\begin{equation}\label{def weighted m-weak core}
A^{{\core}_m,W}=A^{\weak_m,W} P_{(WA)^m},
\end{equation}
is called  the $W$-weighted $m$-weak core inverse of $A$.
\end{definition}

In \cite{MoStKa} it was proved that $W$-weighted $m$-weak group inverse of a rectangular matrix always exists and is unique. So, it is clear from \eqref{def weighted m-weak core} that the $W$-weighted $m$-weak core inverse of a rectangular matrix always exists and is unique.

\begin{remark}\label{remark 1}
\begin{enumerate}[(i)]
\item   If $m=1$, then  $A^{\core_1,W}=A^{\weak,W} P_{WA}$ which is an
extension of the WC inverse to the rectangular case;
\item  If  $m\ge k$ , then the $W$-weighted $m$-weak core inverse coincides with the $W$-weighted core-EP inverse, that is, $A^{\core_m,W}=A^{\odagger,W}$. In fact, when $m\ge k$ we have $A^{\weak_m,W}=A^{d,W}$ and $P_{(WA)^m}=P_{(WA)^k}$. Thus, from the second equality in \eqref{properties weighted CEP}, it follows the affirmation.
\end{enumerate}
\end{remark}

Next, we give different representations for the $W$-weighted $m$-weak core inverse.

\begin{proposition} \label{proposition 1}
Let $A \in \Cm$, $0\neq W\in \Cn$, $k=\max\{\ind(AW), \ind(WA)\}$, and $m\in \mathbb{N}$. Then
\begin{enumerate}[{\rm(a)}]
\item  $A^{\core_m,W}=(A^{\odagger,W}W)^m (AW)^{m-1}AP_{(WA)^m}$.
\item  $A^{\core_m,W}=A[(WA)^{\odagger}]^{m+2}(WA)^mP_{(WA)^m}$.
\item  $A^{\core_m,W}=A(WA)^{\odagger}(WA)^{\weak_m}P_{(WA)^m}$.
\item  $A^{\core_m,W}=A[(WA)^d]^{m+2}P_{(WA)^k}(WA)^mP_{(WA)^m}$.
\item  $A^{\core_m,W}=A(WA)^k[(WA)^{k+m+2}]^\dag (WA)^mP_{(WA)^m}$.
\item $A^{\core_m,W}=A(WA)^k[(WA)^{k+m+2}]^\dag (WA)^{2m}[(WA)^m]^\dag$.
\item  $A^{\core_m,W}=A[(WA)^d]^{m+2}P_{(WA)^k}[((WA)^m)^\diamond]^\dag$.
\item  $A^{\core_m,W}=(AW)^kA[(WA)^{k+m+2}]^\dag (WA)^mP_{(WA)^m}$.
\item  $A^{\core_m,W}=YWAWA^{\core_m,W}$, where $YWAWY=Y$ and $\Ra(Y)=\Ra((AW)^k)$.
\end{enumerate}
\end{proposition}
\begin{proof}
(a) It directly follows from \eqref{def weighted m-weak group} and \eqref{def weighted m-weak core}.\\
Items (b)-(e) follow from \cite[Lemma 2.1]{MoStKa}. \\
(f) The affirmation follows from (e). \\
(g) By (d) and definition of BT inverse of $(WA)^m$ \cite{BaTr2}.  \\
(h) Follows of the fact that $(AW)^kA=A(WA)^k$. \\
(i) Since $Y$ is an outer inverse of $WAW$, it is clear that $\Ra(Y)=\Ra(YWAW)$. Thus, from $\Ra(Y)=\Ra((AW)^k)$ and (h) we have
\[\Ra(A^{\core_m,W})\subseteq \Ra((AW)^k)=\Ra(Y)=\Ra(YWAW)=\Nu(I_n-YWAW),\]
whence $(I_n-YWAW)A^{\core_m,W}=0$. So, $A^{\core_m}=YWAWA^{\core_m}$.
\end{proof}

\begin{remark} Observe that the matrix $Y$ in Proposition \ref{proposition 1} (i) can be one of generalized inverses: $A^{d,W}$, $A^{\odagger,W}$ or $A^{\weak_m,W}$.
\end{remark}

Some more properties of the $W$-weighted $m$-weak core inverse are established below. Before, we provide some auxiliary lemmas.

\begin{lemma}\label{lemma projector 1} Let $A\in \Cnn$,  $\ind(A)=k$, and $m\in \mathbb{N}$. Then for each integer $\ell\ge k$ we have $P_{A^m}A^{\ell}=A^{\ell}$.
\end{lemma}
\begin{proof}
Note that $P_{A^m}A^{\ell}=A^{\ell}$ holds if and only if $\Ra(A^{\ell})\subseteq \Nu(I_n-P_{A^m})=\Ra(P_{A^m})=\Ra(A^m)$, which is always true. In fact, if $\ell \ge m$ the affirmation is trivial. When $\ell < m$, from  $\ell\ge k=\ind(A)$ we obtain $\Ra(A^{\ell})\subseteq \Ra(A^m)$ by definition of index of a matrix.
\end{proof}

\begin{lemma} \label{lemma null space} Let $A$, $B$, and $C$ be complex rectangular matrices of adequate size such that $\Nu(A)=\Nu(B)$. Then $\Nu(AC)=\Nu(BC)$.
\end{lemma}
\begin{proof}
As $\Nu(A)=\Nu(B)$ is equivalent to $\Ra(A^*)=\Ra(B^*)$, we have $\Ra((AC)^*)=C^*\Ra(A^*)=C^*\Ra(B^*)=\Ra((BC)^*)$, whence $\Nu(AC)=\Nu(BC)$.
\end{proof}

As the $W$-weighted $m$-core inverse is defined in term of the $W$-weighted $m$-weak group inverse, in the following lemma we obtain some more properties of this last generalized inverse.

\begin{lemma}\label{lemma more properties of weighted WG} Let $A\in \Cm$, $0\neq W\in \Cn$ and $m\in \mathbb{N}$. Then the following statements hold:
\begin{enumerate}[{\rm (a)}]
\item $A^{\weak_m,W}WAWA^{\weak_m,W}=A^{\weak_m,W}$.
\item $AWA^{\weak_m,W}WA^{\weak_m,W}=A^{\weak_m,W}$.
\item $WA^{\weak_m,W}=(WA)^{\weak_m}$.
\item $A^{\weak_m,W}=A[(WA)^{\weak_m}]^2$.
\end{enumerate}
\end{lemma}
\begin{proof}
(a) By \cite[Lemma 2.2]{MoStKa}.\\
(b) It is due a \eqref{system weighted m-weak group}. \\
(c) From \cite[Remark 2.1]{MoStKa}, we have $WA^{\weak_m,W}=P_{(WA)^k}(WA)^{\weak_m}$, where $k=\max\{\ind(AW), \ind(WA)\}$. Thus, as $\Ra((WA)^{\weak_m})=\Ra((WA)^k)$ we get $P_{(WA)^k}(WA)^{\weak_m}=(WA)^{\weak_m}$. \\
(d) Directly follows from (b) and (c). \\
\end{proof}
\begin{theorem}\label{theorem more properties}
Let $A\in \Cm$, $0\neq W\in \Cn$, $k=\max\{\ind(AW), \ind(WA)\}$, and $m\in \mathbb{N}$. Then the following statements hold:
\begin{enumerate}[{\rm (a)}]
\item $A^{\core_m,W}WAWA^{\core_m,W}=A^{\core_m,W}$.
\item $A^{\core_m,W} W (AW)^{k+1}=(AW)^k$.
\item $\Ra(A^{\core_m, W})= \Ra((AW)^k)$ and $\Nu(A^{\core_m,W})=\Nu([(WA)^k]^*(WA)^m P_{(WA)^m})$.
\item $AWA^{\core_m,W}WA^{\core_m, W}=A^{\core_m,W}$.
\item $AWA^{\core_m,W}=A(WA)^{\core_m}$.
\item $A^{\core_m,W}WA=A^{\weak_m,W}P_{(WA)^m}WA$.
\end{enumerate}
\end{theorem}
\begin{proof}
(a) By Lemma 2.2 in \cite{MoStKa}, we know that $\Ra(WAWA^{\weak_m,W})=\Ra((WA)^k)$. Thus, $WAWA^{\weak_m,W}=(WA)^kB$, for some matrix $B$. Thus, by applying Lemma \ref{lemma projector 1} and Lemma \ref{lemma more properties of weighted WG} (a), we have
\begin{eqnarray*}
A^{\core_m,W}WAWA^{\core_m,W} &=& A^{\weak_m,W}P_{(WA)^m}[WAWA^{\weak_m,W}]P_{(WA)^m} \\
&=& A^{\weak_m,W}P_{(WA)^m}(WA)^k B P_{(WA)^m} \\
&=& A^{\weak_m,W}(WA)^k B P_{(WA)^m} \\
&=&A^{\weak_m,W}WAWA^{\weak_m,W}P_{(WA)^m}\\
&=&A^{\weak_m,W}P_{(WA)^m}\\
&=&A^{\core_m,W}.
\end{eqnarray*}
(b) By Proposition \ref{proposition 1} (d), we know that $A^{\core_m,W}=A[(WA)^d]^{m+2}P_{(WA)^k}(WA)^mP_{(WA)^m}$. So, from Lemma \ref{lemma projector 1} and \eqref{properties w drazin}, we have
\begin{eqnarray*}
A^{\core_m,W}W(AW)^{k+1} &=& A[(WA)^d]^{m+2}P_{(WA)^k}(WA)^mP_{(WA)^m}W(AW)^{k+1} \\
&=& A[(WA)^d]^{m+2}P_{(WA)^k}(WA)^m[P_{(WA)^m}(WA)^{k+1}]W\\
&=& A[(WA)^d]^{m+2}(WA)^{m+k+1}W \\
&=& [A(WA)^d] (WA)^k W \\
&=& AWA^{d,W}(WA)^k W \\
&=& A^{d,W} W A (WA)^k W \\
&=& A^{d,W} W (AW)^{k+1}\\
&=& (AW)^k.
\end{eqnarray*}
(c) The equality $\Ra(A^{\core_m, W})=\Ra((AW)^k)$ follows from (b) and Proposition \ref{proposition 1} (h). On the other hand, from \cite[Lemma 2.2]{MoStKa} we know that $\Nu(A^{\weak_m,W})=\Nu([(WA)^k]^*(WA)^m)$. Thus, from \eqref{def weighted m-weak core} and Lemma \ref{lemma null space} we deduce $\Nu(A^{\core_m,W})=\Nu([(WA)^k]^*(WA)^m P_{(WA)^m})$.\\
(d) As $(WA)^{\weak_m}=(WA)^kB$ for some matrix $B$, from Lemma \ref{lemma projector 1} and parts (c) and (d) of Lemma \ref{lemma more properties of weighted WG} we obtain
\begin{eqnarray*}
AWA^{\core_m,W}WA^{\core_m,W} &=& A[WA^{\weak_m,W}]P_{(WA)^m}[WA^{\weak_m,W}]P_{(WA)^m} \\
&=& A (WA)^{\weak_m} P_{(WA)^m} (WA)^{\weak_m} P_{(WA)^m} \\
&=& A (WA)^{\weak_m} P_{(WA)^m} (WA)^k B P_{(WA)^m} \\
&=& A [(WA)^{\weak_m}]^2 P_{(WA)^m} \\
&=& A^{\weak_m,W}P_{(WA)^m} \\
&=& A^{\core_m,W}.
\end{eqnarray*}
(e) From \eqref{def weighted m-weak core} and Lemma \ref{lemma more properties of weighted WG} (c) we have \[AWA^{\core_m,W}=A[WA^{\weak_m}]P_{(WA)^m}=A(WA)^{\weak_m}P_{(WA)^m}=A(WA)^{\core_m}.\]
Statement (f) directly follows from \eqref{def weighted m-weak core}.
\end{proof}

We finish this section with a result that shows that the $W$-weighted $m$-weak core inverse can be written as  a generalized inverse with prescribed range and null space. Also,  some special idempotent matrices determined by such a inverse  are considered.
Recall that the outer inverse of a matrix $A\in \Cm$
which is uniquely determined by the null space $S$ and the column space $T$ is denoted by $A^{(2)}_{T,S}=X\in \Cn$ and satisfies 
$XAX=X$, $\Nu(X)=S$, and $\Ra(X)=T$,
where $s\le r=\ra(A)$ is  dimension of the subspace $T\subseteq \C^n$ and $p-s$ is  dimension of the subspace $S\subseteq \C^p$.

\begin{proposition}\label{proposition 2} Let $A\in \Cm$, $0\neq W\in \Cn$, $k=\max\{\ind(AW), \ind(WA)\}$, and $m\in \mathbb{N}$. Then the following representations are valid:
\begin{enumerate}[\rm (a)]
\item $A^{\core_m,W}=(WAW)^{(2)}_{\Ra(((AW)^k)),\,\Nu([(WA)^k]^*(WA)^m P_{(WA)^m})}$.
\item $WAWA^{\core_m,W}=P_{\Ra((WA)^k),\,\Nu([(WA)^k]^*(WA)^m P_{(WA)^m})}$.
\item $A^{\core_m,W}WAW=P_{\Ra((AW)^k),\,\Nu([(WA)^k]^*(WA)^m P_{(WA)^m}WAW)}$.
\item $A^{\core_m,W}=(AW)^k ( [(WA)^k]^* (WA)^{m+k+1} W)^\dag [(WA)^k]^* (WA)^{2m} [(WA)^m]^\dag$.
\end{enumerate}
\end{proposition}
\begin{proof}
(a) It follows from parts (a) and (c) of Theorem \ref{theorem more properties}.\\
(b) Clearly, Theorem \ref{theorem more properties} (a) implies that  $WAWA^{\core_m, W}$ is idempotent.  Moreover, from Theorem \ref{theorem more properties} (c) we obtain
\begin{eqnarray*}
\Ra(WAWA^{\core_m, W})&=& WAW\Ra(A^{\core_m, W})\\
&=& WAW\Ra((AW)^k) \\
&=& \Ra((WA)^{k+1}W),
\end{eqnarray*}
whence $\Ra(WAWA^{\core_m, W})\subseteq \Ra((WA)^k)$. Now, by Theorem \ref{theorem more properties} (b) we get $A^{\core_m,W} W (AW)^{k+1}=(AW)^k$ which implies $WAWA^{\core_m,W} W (AW)^{k+1}A=WAW(AW)^kA=(WA)^{k+2}$. Thus, $\Ra((WA)^k)=\Ra((WA)^{k+2})\subseteq \Ra(WAWA^{\core_m, W})$.\\
On the other hand, as $A^{\core_m, W}$ is an outer inverse of $WAW$, it is well known that
$\Nu(WAWA^{\core_m, W})=\Nu(A^{\core_m, W})$. Thus,  Theorem \ref{theorem more properties} (c)  yield to $\Nu(WAWA^{\core_m, W})=\Nu([(WA)^k]^*(WA)^m P_{(WA)^m})$. \\
(c) Since $A^{\core_m, W}WAWA^{\core_m, W}=A^{\core_m, W}$, clearly $\Ra(A^{\core_m,W}WAW)=\Ra(A^{\core_m,W})=\Ra((AW)^k)$. \\
The fact that $\Nu(A^{\core_m,W}WAW)=\Nu([(WA)^k]^*(WA)^m P_{(WA)^m}WAW)$ follows from Lemma \ref{lemma null space} and Theorem \ref{theorem more properties} (c).\\
(d) This expression follows by part (a) and Urquhart formula \cite{BeGr}.
\end{proof}

\begin{corollary} Let $A\in \Cnn$, $k=\ind(A)$, and $m\in \mathbb{N}$. Then the following statements hold:
\begin{enumerate}[\rm (a)]
\item $A^{\core_m}=A^{(2)}_{\Ra(A^k),\,\Nu((A^k)^*A^mP_{A^m})}$.
\item $AA^{\core_m}=P_{\Ra(A^k),\,\Nu((A^k)^*A^mP_{A^m})}$.
\item $A^{\core_m}A=P_{\Ra(A^k),\,\Nu((A^k)^*A^mP_{A^m}A)}$.
\item $A^{\core_m}=A^k ( (A^k)^* A^{m+k+1})^\dag (A^k)^* A^{2m} (A^m)^\dag$.
\end{enumerate}
\end{corollary}
\begin{proof}
(a)-(d) immediately follow from Proposition \ref{proposition 2} by taking $p=n$ and $W=I_n$.
\end{proof}

\section{Algebraic characterizations of the $W$-weighted $m$-core inverse}

In this section, we give some algebraic characterizations of the $W$-weighted $m$-weak core inverse.

\begin{theorem} \label{characterization 1} Let $A \in \Cm$, $0\neq W\in \Cn$, $k=\max\{\ind(AW), \ind(WA)\}$, and $m\in\mathbb{N}$. Then  the system of equations
\begin{equation}\label{system 1}
XWAWX=X, \quad AWX=A(WA)^{\core_m},\quad XWA=A^{\weak_m,W}P_{(WA)^m} WA,
\end{equation}
is consistent and its unique solution is the matrix $X=A^{\core_m,W}$.
\end{theorem}
\begin{proof}
\emph{Existence.} Let $X:=A^{\core_m,W}$. From parts (a), (e) and (f) of Theorem \ref{theorem more properties}, it is easy to see that $X$ satisfies  all the  equations in \eqref{system 1}. \\
\emph{Uniqueness.} We assume the $X_1$ and $X_2$ are two solutions of the system of equations in (\ref{system 1}). Then 
\[X_2=(X_2WA)X_2=(X_1WA)WX_2=X_1W(AWX_2)=X_1WAWX_1=X_1.\]
 \end{proof}

\begin{theorem}\label{characterization 2} Let $A \in \Cm$, $0\neq W\in \Cn$, $k=\max\{\ind(AW), \ind(WA)\}$, and $m\in\mathbb{N}$. The system given by
\begin{equation}\label{system 2}
AWX=A(WA)^{\core_m}, \quad \Ra(X)\subseteq\Ra((AW)^k),
\end{equation}
is consistent and its unique solution is $X=A^{\core_m,W}$.
\end{theorem}
\begin{proof}
\emph{Existence.} Let $X:=A^{\core_m}$. Clearly, from parts (c) and (e) of Theorem \ref{theorem more properties}, $X$ satisfies both conditions in \eqref{system 2}.  \\
\emph{Uniqueness.}  Let $X_1$ and $X_2$ be two solutions of the system \eqref{system 2},
that is, \[AWX_1=AWX_2=A(WA)^{\core_m}, \quad \Ra(X_1)\subseteq \Ra((AW)^k), \quad \Ra(X_2)\subseteq\Ra((AW)^k).\]
Thus,   $\Ra(X_1-X_2)\subseteq \Nu(AW)\subseteq \Nu((AW)^k)$ and
$\Ra(X_1-X_2)\subseteq \Ra((AW)^k).$ Therefore, $\Ra(X_1-X_2)\subseteq \Ra((AW)^k)\cap \Nu((AW)^k)=\{0\}$. Thus, $X_1=X_2.$
\end{proof}

The first equation in \eqref{system 2} can be weakened by the condition $WAWX=AW(AW)^{\core_m}$.

\begin{theorem}\label{characterization 2 bis} Let $A \in \Cm$, $0\neq W\in \Cn$, $k=\max\{\ind(AW), \ind(WA)\}$, and $m\in\mathbb{N}$. The system given by
\begin{equation}\label{system 2 bis}
WAWX=WA(WA)^{\core_m}, \quad \Ra(X)\subseteq\Ra((AW)^k),
\end{equation}
is consistent and its unique solution is $X=A^{\core_m,W}$.
\end{theorem}
\begin{proof}
It is similar to the proof of the Theorem \ref{characterization 2}.
\end{proof}

\begin{remark} When $m\ge k$ in Theorem \ref{characterization 2 bis} , we recover the characterization for the $W$-weighted core-EP inverse given in \eqref{w cep system}. In fact, if $m\ge k$ then $(WA)^{\core_m}=(WA)^{\odagger}$. Thus, $WAWX=WA(WA)^{\core_m}=WA(WA)^{\odagger}=P_{(WA)^k}$.
\end{remark}

\begin{theorem}\label{characterization 3} Let $A \in \Cm$, $0\neq W\in \Cn$, $k=\max\{\ind(AW), \ind(WA)\}$, and $m\in\mathbb{N}$. The system given by
\begin{equation}\label{system 3}
AWX=A(WA)^{\core_m}, \quad AWXWX=X,
\end{equation}
is consistent and its unique solution is $X=A^{\core_m,W}$.
\end{theorem}
\begin{proof}
\emph{Existence.} Let $X:=A^{\core_m,W}$. From parts (c) and (d) of Theorem \ref{theorem more properties} we deduce that $X$ satisfies  the two equations in \eqref{system 3}.  \\
\emph{Uniqueness.}  Let $X$ be an arbitrary solution of the system \eqref{system 3}. Note that $AWXWX=X$ implies $X=(AW)^k (XW)^kX$, whence $\Ra(X)\subseteq \Ra((AW)^k)$. Now, the uniqueness is due to Theorem \ref{characterization 2}.
\end{proof}

As a consequence of above results, we obtain characterizations of the  $m$-weak core inverse of a complex square matrix.

\begin{corollary}\label{corollary square 2}
Let $A \in \Cnn$, $\ind(A)=k$, and $m\in \mathbb{N}$. Then the following statements are equivalent:
\begin{enumerate}[{\rm(a)}]
\item $X$ is the  $m$-weak core inverse $A^{\core_m}$ of $A$;
\item $XAX=X$, $AX=(A^{\odagger})^m A^m P_{A^m}$, and $XA=(A^{\odagger})^{m+1}A^m P_{A^m}A$;
\item $AX=(A^{\odagger})^m A^m P_{A^m}$ and $\Ra(X)\subseteq \Ra(A^k)$;
\item $AX=(A^{\odagger})^m A^m P_{A^m}$ and $AX^2=X$.
\end{enumerate}
\end{corollary}

\section{Canonical form of the $W$-weighted $m$-weak core inverse}

As proved in \cite{Wang}, every matrix $A\in \Cnn$ of index $k$ admits a core-EP decomposition. The canonical form of a matrix  when is represented by its core-EP decomposition is given by
\begin{equation} \label{core EP decomposition}
A= U\left[\begin{array}{cc}
T & S\\
0 & N
\end{array}\right]U^*,
\end{equation}
where $U\in \Cnn$ is unitary, $T\in \mathbb{C}^{t\times t}$ is a nonsingular matrix with $t:=\ra(T)=\ra(A^k)$, and $N\in \mathbb{C}^{(n-t)\times (n-t)}$ is nilpotent of index $k$.

In \cite{FeLeTh1} the authors introduced a simultaneous unitary block upper triangularization of a pair of rectangular matrices generalizing the core-EP decomposition to the rectangular case. More precisely,  we have the following result:
\begin{theorem}{\rm\cite{FeLeTh1}}\label{decomposition}
Let $A\in\Cm$ and $0\neq W\in \Cn$ with  $k=\max\{\ind(AW), \ind(WA)\}$.
 Then there exist two unitary matrices
 $U \in \Cmm$, $V \in \Cnn$, two nonsingular matrices $A_1, W_1 \in \mathbb{C}^{t \times t}$, and two matrices $A_3 \in
\mathbb{C}^{(m-t)\times(n-t)}$ and $W_3 \in
\mathbb{C}^{(n-t)\times(m-t)}$ such that $A_3W_3$ and $W_3A_3$ are
nilpotent of indices $\ind(AW)$ and  $\ind(WA)$,
respectively, with
\begin{equation} \label{weighted CEP decomposition}
A = U\left[\begin{array}{cc}
A_1 & A_2 \\
0 & A_3
\end{array}\right]V^* \quad \text{and} \quad
W = V\left[\begin{array}{cc}
W_1 & W_2 \\
0 & W_3
\end{array}\right]U^*.
\end{equation}
\end{theorem}

\begin{remark} \label{remark AW and WA}
If $A$ and $W$ are written as in (\ref{weighted CEP decomposition}). Then for each $\ell\in \mathbb{N}$ we have
\begin{equation}\label{AW^l and WA^l}
(AW)^\ell = U\left[\begin{array}{cc}
(A_1W_1)^\ell & \widetilde{S}_\ell\\
0 & (A_3W_3)^{\ell}
\end{array}\right]U^*, \quad (WA)^\ell = U\left[\begin{array}{cc}
(W_1A_1)^\ell & \widetilde{T}_\ell\\
0 & (W_3A_3)^{\ell}
\end{array}\right]U^*,
\end{equation}
where
\[\widetilde{S}_\ell:=\sum\limits_{j=0}^{{\ell}-1} (A_1W_1)^{j} (A_1W_2+A_2W_3) (A_3W_3)^{{\ell}-1-j}, \quad \widetilde{T}_\ell:=\sum\limits_{j=0}^{{\ell}-1} (W_1A_1)^{j} (W_1A_2+W_2A_3) (W_3A_3)^{{\ell}-1-j}.\]
The expressions given in \eqref{AW^l and WA^l} are the core-EP decomposition of $(AW)^\ell$ and $(WA)^\ell$, respectively.
In particular, if $\ell\ge k$ then $(A_3W_3)^{\ell}=0$ and $(W_3A_3)^{\ell}=0$ in \eqref{AW^l and WA^l}.
\end{remark}

The  $W$-weighted core-EP inverse of a rectangular matrix can be represented  by using the simultaneous decomposition given in \eqref{weighted CEP decomposition}:

\begin{equation}\label{canonical weighted CEP}
 A^{\odagger,W}= U\left[\begin{array}{cc}
(W_1A_1W_1)^{-1} & 0 \\
0 & 0
\end{array}\right]V^*.
\end{equation}
In \cite{FeLeTh1}, the authors also gave the following useful representations:
\begin{equation}\label{canonical form of AW and WA}
 (AW)^{\odagger}=U\left[\begin{array}{cc}
(A_1W_1)^{-1} & 0 \\
0 & 0
\end{array}\right]U^*, \quad (WA)^{\odagger}=V\left[\begin{array}{cc}
(W_1A_1)^{-1} & 0 \\
0 & 0
\end{array}\right]V^*.
\end{equation}

\begin{lemma}{\rm\cite{FeThTo}}
Let
$A=U\begin{bmatrix}
        A_1 & A_2 \\
        0 & A_3 \\
      \end{bmatrix}V^* \in \Cm$ be  such that  $A_1 \in \C^{t\times t}$ is nonsingular and $U \in {\mathbb C}^{m \times m}$ and $V \in {\mathbb C}^{n \times n}$ are unitary. Then
\begin{equation} \label{MP triangular}
A^\dag = V\left[\begin{array}{cc}
A_1^*\Omega & -A_1^*\Omega A_2 A_3^\dagger\\
(I_{n-t}-Q_{A_3})A_2^*\Omega & A_3^\dagger-(I_{n-t}-Q_{A_3})A_2^*\Omega A_2 A_3^\dagger
\end{array}\right]U^*,
\end{equation}
where $\Omega=(A_1 A_1^*+ A_2(I_{n-t}-Q_{A_3})A_2^*)^{-1}$.
In consequence,
\begin{equation} \label{projector}
P_A= U\left[\begin{array}{cc}
I_t & 0\\
0 & P_{A_3}
\end{array}\right]U^*.
\end{equation}
\end{lemma}

Now, we present a canonical form for the $W$-weighted $m$-weak core inverse inverse by using the weighted core-EP decomposition.

\begin{theorem}\label{canonical form weighted m-weak core}
Let $A \in \Cm$, $0\neq W\in \Cn$, $k=\max\{\ind(AW), \ind(WA)\}$, and $m\in \mathbb{N}$. If $A$ and $W$ are written  as in  (\ref{weighted CEP decomposition}), then  the $W$-weighted $m$-weak core inverse is given by
\begin{eqnarray}
A^{\core_m,W} &=& U\left[\begin{array}{cc}
(W_1A_1W_1)^{-1} &  (A_1W_1)^{-(m+1)}W_1^{-1}\widetilde{T}_m P_{(W_3A_3)^m} \\
0 & 0
\end{array}\right]V^* \label{canonical form 1}\\
&=& U\left[\begin{array}{cc}
(W_1A_1W_1)^{-1} &  W_1^{-1}(W_1A_1)^{-(m+1)}\widetilde{T}_m P_{(W_3A_3)^m} \\
0 & 0
\end{array}\right]V^*, \label{canonical form 2}
\end{eqnarray}
where
$\widetilde{T}_m=\sum\limits_{j=0}^{m-1} (W_1A_1)^j (W_1A_2+W_2 A_3) (W_3A_3)^{m-1-j}$.
\end{theorem}
\begin{proof}
From Proposition \ref{proposition 1} (b), we have that $A^{\core_m,W}=A[(WA)^{\odagger}]^{m+2}(WA)^mP_{(WA)^m}$. So, from Remark \ref{remark AW and WA}, \eqref{canonical form of AW and WA} and \eqref{projector} we obtain
\begin{eqnarray*}
A^{\core_m,W} &=& U\begin{bmatrix}
 A_1 & A_2 \\
  0 & A_3
  \end{bmatrix} \begin{bmatrix}
 (W_1A_1)^{-(m+2)} & 0 \\
  0 & 0
  \end{bmatrix}\begin{bmatrix}
 (W_1A_1)^m & \widetilde{T}_m \\
  0 & (W_3A_3)^m
  \end{bmatrix}\left[\begin{array}{cc}
I_t &  0 \\
0 & P_{(W_3A_3)^m}
\end{array}\right]V^* \\
&=& U\begin{bmatrix}
 A_1(W_1A_1)^{-2} & W_1^{-1}(W_1A_1)^{-(m+1)}\widetilde{T}_m P_{(W_3A_3)^m}\\
  0 & 0
  \end{bmatrix}V^*,\\
\end{eqnarray*}
which lead to \eqref{canonical form 1} and \eqref{canonical form 2} by using the identity $W_1(A_1W_1)^{m+1}=(W_1A_1)^{m+1}W_1$.
\end{proof}

\begin{remark} When $m=\ind(WA)$ it is clear that $P_{(W_3A_3)^m}=0$. Thus, from \eqref{canonical weighted CEP} and Theorem \ref{canonical form weighted m-weak core} we deduce that $A^{\core_m,W}=A^{\odagger,W}$.
\end{remark}

In the following example we show that when $1<m<k$ with $m\neq Ind(WA)$, this new inverse is different from other known ones.

\begin{example}
Consider the matrices
\[
A=\left[\begin{array}{cccc}
1 &  0 & 0 & 0 \\
0 &  1 & 0 & 0 \\
0 &  0 & 0 & 1
\end{array}\right] \quad \text{and}  \quad W=\left[\begin{array}{ccc}
0 &  0 & 1 \\
0 &  1 & 1 \\
1 &  1 & 1 \\
0 &  0 & 0
\end{array}\right] .\]
As $\ind(AW)=2$ and $\ind(WA)=3$, we have $k=\max\{\ind(AW), \ind(WA)\}=3$. Therefore, we must consider $m=2$. Thus, the $W$-weighted Drazin inverse, the $W$-weighted core-EP inverse, and the $W$-weighted $m$-weak core inverse are given by
\[
A^{d,W} =
\left[\begin{array}{cccc}
0 &  0 & 0 & 0 \\
0 &  1 & 0 & 1 \\
0 &  0 & 0 & 0
\end{array}\right], \quad
A^{\odagger,W}=
\left[\begin{array}{cccc}
0 &  0 & 0 & 0 \\
0 &  \frac{1}{2} & \frac{1}{2} & 0 \\
0 &  0 & 0 & 0
\end{array}\right], \quad A^{\core_2,W}=
\left[\begin{array}{cccc}
0 &  0 & 0 & 0 \\
0 &  1 & 0 & 0 \\
0 &  0 & 0 & 0
\end{array}\right]. \]

\end{example}

\begin{lemma} \label{lemma canonical form m-weak core of AW and WA}
Let $A \in \Cm$, $0\neq W\in \Cn$, $k=\max\{\ind(AW), \ind(WA)\}$, and $m\in \mathbb{N}$. If $A$ and $W$ are written  as in  (\ref{weighted CEP decomposition}),
then results that
\begin{enumerate}[(a)]
\item $(AW)^{\core_m}=U\left[\begin{array}{cc}
(A_1W_1)^{-1} & (A_1W_1)^{-(m+1)}\widetilde{S}_m P_{(A_3W_3)^m}  \\
0 & 0
\end{array}\right]U^*.$

\item $(WA)^{\core_m}=V\left[\begin{array}{cc}
(W_1A_1)^{-1} & (W_1A_1)^{-(m+1)}\widetilde{T}_m P_{(W_3A_3)^m} \\
0 & 0
\end{array}\right]V^*.$
\end{enumerate}
\end{lemma}
\begin{proof} We only prove part (a) since the proof of (b) is analogous.
From Remark \ref{remark AW and WA}, we know that
\begin{equation}\label{wa}
AW= V\left[\begin{array}{cc}
A_1W_1 & A_1W_2 + A_2W_3\\
0 & A_3W_2
\end{array}\right]V^*,
\end{equation}
is a core-EP decomposition of $AW$. Now, by applying \cite[Theorem 4.6]{FeMa5}, we get (a).
\end{proof}

\begin{remark} Note that the expressions obtained in the previous lemma coincide with those given in \eqref{canonical form of AW and WA} when $m\ge k$ or $m=\ind(WA)$.
\end{remark}
From Theorem \ref{canonical form weighted m-weak core} and Lemma \ref{lemma canonical form m-weak core of AW and WA}, we can derive easily more properties for the $W$-weighted $m$-weak core inverse.

\begin{theorem}\label{more properties weighted m-weak core} Let $A\in \Cm$, $0\neq W\in \Cn$ and $m\in \mathbb{N}$. Then the following statements hold:
\begin{enumerate}[{\rm (a)}]
\item $WA^{\core_m,W}=(WA)^{\core_m}$.
\item $A^{\core_m,W}=AWA^{\core_m,W}WA^{\core_m,W}$.
\item $A^{\core_m,W}=(AW)^{\core_m}A(WA)^{\core_m}$.
\item $A^{\core_m,W}=A[(WA)^{\core_m}]^2$.
\end{enumerate}
\end{theorem}

Recall from \eqref{properties w drazin} that the $W$-weighted Drazin inverse satisfies the interesting identities $A^{d,W}=[(AW)^d]^2A=A[(WA)^d]^2$, $A^{d,W}W=(AW)^d$, and $WA^{d,W}=(WA)^d$. By Theorem \ref{more properties weighted m-weak core} we know that some of these properties remain valid for the $W$-weighted $m$-weak core inverse .
In the following example we show that in general $A^{\core_m,W}\neq [(AW)^{\core_m}]^2A$ and $A^{\core_m,W}W\neq (AW)^{\core_m}$.

\begin{example}\label{ex1-W-mWC} Let
\[A=\left[\begin{array}{ccc}
1 &  1 & 0 \\
0 &  1 & 0 \\
0 &  0 & 1 \\
0 &  0 & 0
\end{array}\right] \quad \text{and}  \quad W=\left[\begin{array}{cccc}
1 &  0 & 1 & 0 \\
0 &  0 & 1 & 0 \\
0 &  0 & 0 & 1
\end{array}\right].\]
Note that $k=\max\{{\rm Ind}(AW), {\rm Ind}(WA)\}=\max\{3,2\}=3$. For  $m=2$, from Proposition \ref{proposition 1} (f)  we obtain
\[A^{\core_2,W}=A(WA)^3 [(WA)^7]^\dag (WA)^4[(WA)^2]^\dag=
\left[\begin{array}{ccc}
1 &  0 & 0 \\
0 &  0 & 0 \\
0 &  0 & 0 \\
0 &  0 & 0
\end{array}\right].\]
Also, we {\rm\cite[Theorem 4.7 (f)]{FeMa5}} we have
\[
(AW)^{\core_2}=(AW)^3[(AW)^{6}]^\dag (AW)^4 [(AW)^2]^\dag=
\left[\begin{array}{cccc}
1 &  0 & 0 & 0 \\
0 &  0 & 0 & 0 \\
0 &  0 & 0 & 0 \\
0 &  0 & 0 &0
\end{array}\right].\]
Thus, a straightforward computation shows that
\[A^{\core_2,W}\neq [(AW)^{\core_2}]^2A=\left[\begin{array}{ccc}
1 &  1 & 0 \\
0 &  0 & 0 \\
0 &  0 & 0 \\
0 &  0 & 0
\end{array}\right] \quad \text{and}\quad (AW)^{\core_2}\neq A^{\core_2,W}W=\left[\begin{array}{cccc}
1 &  0 & 1 & 0 \\
0 &  0 & 0 & 0 \\
0 &  0 & 0 & 0 \\
0 &  0 & 0 &0
\end{array}\right].\]
\end{example}

\section{Applications}

The special solution of matrix equations has raised much interest among researchers due
to the wide applications such as robust control, neural network, singular system control,
model reduction and image processing. 
To find appropriate approximations to inconsistent system of linear equations $Ax=b$, one typical approach is to asks for, so called,
generalized solutions, defined as solutions to $GAx=Gb$ with respect to an appropriate matrix $G$ \cite{Mas}.
This approach has been exploited extensively.
One particular choice is $G=A^*$, which leads to widely used least-squares solutions obtained as solutions to the normal equation $A^* Ax=A^* b$.
Another important choice is $G=A^k$ and $k=\ind(A)$, which leads to the so called Drazin normal equation $A^{k+1}x=A^kb$ and usage of the Drazin inverse solution $A^Db$.

Applying the $W$-weighted $m$-weak core inverse, we obtain solvability of certain systems of linear matrix equations.

\begin{theorem}\label{te11-W-mWC} Let $A \in \Cm$, $0\neq W\in \Cn$, $k=\max\{\ind(AW), \ind(WA)\}$, $m\in \mathbb{N}$, and $b\in\C^n$. The general solution to the equation
\begin{equation}[(WA)^k]^*(WA)^{m+1}Wx=[(WA)^k]^*(WA)^{2m}[(WA)^m]^\dag b,\label{jed1-W-mWC}\end{equation}
is given as
\begin{equation}x=A^{\core_m,W} b+(I_p-A^{\weak_m,W}WAW)y,\label{jed2-W-mWC}\end{equation} for arbitrary $y\in\C^p$.
\end{theorem}

\begin{proof} Assume that $x$ is expressed by (\ref{jed2-W-mWC}). Using Proposition \ref{proposition 1} (f), it follows the expression $A^{\core_m,W}=A(WA)^k[(WA)^{k+m+2}]^\dag (WA)^{2m}[(WA)^m]^\dag$. Thus, as $P_{(WA)^k}=P_{(WA)^\ell}$ for each integer $\ell\ge k$, we obtain
\begin{eqnarray}\label{jed3-W-mWC}
[(WA)^k]^*(WA)^{m+1}WA^{\core_m,W}&=&[(WA)^k]^*P_{(WA)^{k+m+2}}(WA)^{2m}[(WA)^m]^\dag\nonumber\\
&=&[(WA)^k]^*P_{(WA)^k}(WA)^{2m}[(WA)^m]^\dag\nonumber\\
&=&[(WA)^{k}]^*(WA)^{2m}[(WA)^m]^\dag.
\end{eqnarray}
Similarly, as $A^{\weak_m,W}=A(WA)^k[(WA)^{k+m+2}]^\dag (WA)^m$ by \cite[Lemma 2.1]{MoStKa}, we deduce
\begin{eqnarray}\label{jed4-W-mWC}
[(WA)^k]^*(WA)^{m+1}WA^{\weak_m,W}WAW&=&[(WA)^k]^*P_{(WA)^{k+m+2}}(WA)^{m+1}W\nonumber\\
&=& ([(WA)^k]^*P_{(WA)^k})(WA)^{m+1}W\nonumber\\
&=&[(WA)^{k}]^*(WA)^{m+1}W.
\end{eqnarray}
In consequence, the equalities \eqref{jed2-W-mWC}, \eqref{jed3-W-mWC} and \eqref{jed4-W-mWC} imply that $x$ is a solution to
(\ref{jed1-W-mWC}). In fact,
\begin{eqnarray*}
& &[(WA)^k]^*(WA)^{m+1}Wx \\
&=&[(WA)^k]^*(WA)^{m+1}WA^{\core_m,W} b
+[(WA)^k]^*(WA)^{m+1}W(I_p-A^{\weak_m,W}WAW)y\\
&=&[(WA)^k]^*(WA)^{2m}[(WA)^m]^\dag b.
\end{eqnarray*}
On the other hand, let $x$ be a solution to the equation (\ref{jed1-W-mWC}). Since
\begin{eqnarray*}
A^{\core_m,W} b&=&A(WA)^k[(WA)^{k+m+2}]^\dag (WA)^{2m}[(WA)^m]^\dag b\\
&=&A[(WA)^d]^{k+m+2}P_{(WA)^{k+m+2}} (WA)^{2m}[(WA)^m]^\dag b\\
&=& A[(WA)^d]^{k+m+2}P_{(WA)^k} (WA)^{2m}[(WA)^m]^\dag b\\
&=& A[(WA)^d]^{k+m+2}[((WA)^k)^\dag]^*[(WA)^k]^* (WA)^{2m}[(WA)^m]^\dag b\\
&=& A[(WA)^d]^{k+m+2}[((WA)^k)^\dag]^* [(WA)^k]^*(WA)^{m+1}Wx \\
&=& A[(WA)^d]^{k+m+2}P_{(WA)^{k+m+2}}(WA)^{m+1}Wx \\
&=&A(WA)^k[(WA)^{k+m+2}]^\dag (WA)^{m+1}Wx\\
&=&A^{\weak_m,W} WAWx,
\end{eqnarray*}
then $x=A^{\core_m,W} b+(I_p-A^{\weak_m,W}WAW)x$, i.e., the form of $x$ is (\ref{jed2-W-mWC}).
\end{proof}

If we suppose that $b\in\Ra((WA)^m)$ in Theorem \ref{te11-W-mWC}, then, by $(WA)^{m}[(WA)^m]^\dag b=b$, the general solution
to the equation
\begin{equation}[(WA)^k]^*(WA)^{m+1}Wx=[(WA)^k]^*(WA)^{m}b\label{jed2'-W-mWC}\end{equation}  is given as
$x=A(WA)^k[(WA)^{k+m+2}]^\dag (WA)^{m} b+(I_p-A^{\weak_m,W}WAW)y$, for arbitrary $y\in\C^p$.
Remark that, the equation (\ref{jed2'-W-mWC}) has a form $GWAWx=Gb$, for $G=[(WA)^k]^*(WA)^{m}$.

For $p=n$ and $W=I_n$ in Theorem \ref{te11-W-mWC}, we get the next consequence for the $m$-weak core inverse.

\begin{corollary} Let $A \in \Cnn$, $\ind(A)=k$, $m\in \mathbb{N}$, and $b\in\C^n$. The general solution to the equation
\begin{equation}(A^k)^*A^{m+1}x=(A^k)^*A^{2m}(A^m)^\dag b\label{jed1'-W-mWC}\end{equation} is given as
$$x=A^{\core_m} b+(I_n-A^{\weak_m}A)y,$$ for arbitrary $y\in\C^p$.
\end{corollary}

We present conditions under which a solution to (\ref{jed1-W-mWC}) is unique.

\begin{theorem}\label{te11'-W-mWC} Let $A \in \Cm$, $0\neq W\in \Cn$, $k=\max\{\ind(AW), \ind(WA)\}$, $m\in \mathbb{N}$, and $b\in\C^n$. Then $A^{\core_m,W}b$ is the unique solution to {\rm(\ref{jed1-W-mWC})} in $\Ra((AW)^k)$.
\end{theorem}

\begin{proof} By Theorem \ref{theorem more properties} (c) we know that $\Ra(A^{\core_m,W})=\Ra((AW)^k)$. Thus, according to Theorem \ref{te11-W-mWC}, the equation (\ref{jed1-W-mWC}) has a solution $x=A^{\core_m,W}b\in\Ra((AW)^k)$.
In order to prove the uniqueness of the solution, let us consider
two solutions $x_1,x_2\in\Ra((AW)^k)$ to (\ref{jed1-W-mWC}). Therefore,
as $[(WA)^k]^*(WA)^{m+1}Wx_1=[(WA)^k]^*(WA)^{m+1}Wx_2$, from \cite[Lemma 2.2]{MoStKa} we deduce
\begin{eqnarray*}
x_1-x_2 &\in &\Nu([(WA)^k]^*(WA)^{m+1}W)\cap \Ra((AW)^k)\\
&\subseteq & \Nu(A^{\weak_m,W}WAW)\cap \Ra(A^{\weak_m,W}WAW)\\
&= & \{0\}.
\end{eqnarray*}
Hence, $A^{\core_m,W}b=x=x_1$ is
unique solution to (\ref{jed1-W-mWC}) in $\Ra((AW)^k)$.
\end{proof}

Consequently, we get when the equation (\ref{jed1'-W-mWC}) has unique determined solution.

\begin{corollary} Let $A \in \Cnn$, $\ind(A)=k$, $m\in \mathbb{N}$, and $b\in\C^n$. Then $A^{\core_m}b$ is unique solution in $\Ra(A^k)$ to {\rm(\ref{jed1'-W-mWC})}.
\end{corollary}

Similarly as Theorem \ref{te11-W-mWC} and Theorem \ref{te11'-W-mWC}, we can prove the solvability of the next equation.

\begin{theorem}\label{te11''-W-mWC} Let $A \in \Cm$, $0\neq W\in \Cn$, $k=\max\{\ind(AW), \ind(WA)\}$, $m\in \mathbb{N}$, and $b\in\C^n$. The general solution to the equation
\begin{equation}[(WA)^k]^*(WA)^{2m}[(WA)^m]^\dag WAWx=[(WA)^k]^*(WA)^{2m}[(WA)^m]^\dag b,\label{jed1''-W-mWC}\end{equation}
is given as
$$x=A^{\core_m,W} b+(I_p-A^{\core_m,W}WAW)y,$$ for arbitrary $y\in\C^p$. Furthermore, $A^{\core_m,W}b$ is the unique solution to {\rm(\ref{jed1''-W-mWC})} in $\Ra((AW)^k)$.
\end{theorem}

The equation \eqref{jed1''-W-mWC} can be considered as $GWAWx=Gb$ for $G=[(WA)^k]^*(WA)^{2m}[(WA)^m]^\dag$.

\begin{corollary} Let $A \in \Cnn$, $\ind(A)=k$, $m\in \mathbb{N}$, and $b\in\C^n$. The general solution to the equation
\begin{equation}(A^k)^*A^{2m}(A^m)^\dag Ax=(A^k)^*A^{2m}(A^m)^\dag b\label{jed1'''-W-mWC}\end{equation} is given as
$$x=A^{\core_m} b+(I_n-A^{\core_m}A)y,$$ for arbitrary $y\in\C^p$. Furthermore, $A^{\core_m}b$ is the unique solution to {\rm(\ref{jed1'''-W-mWC})} in $\Ra(A^k)$.
\end{corollary}

Also, notice that \eqref{jed1''-W-mWC} has a form $GAx=Gb$ for $G=(A^k)^*A^{2m}(A^m)^\dag $.

In the case that unknown matrix acts on the left, we obtain solvability of the following equations.

\begin{theorem}\label{te12-W-mWC} Let $A \in \Cm$, $0\neq W\in \Cn$, $k=\max\{\ind(AW), \ind(WA)\}$, $m\in \mathbb{N}$, and $B\in\C^{p\times p}$. The general solution to the equation
\begin{equation}XW(AW)^{k+1}=B(AW)^k\label{jed5-W-mWC}\end{equation} is given as
\begin{equation}X=BA^{\core_m,W}+Y(I_n-WAW A^{\core_m,W}),\label{jed6-W-mWC}\end{equation} for arbitrary $Y\in\C^{p\times p}$.
\end{theorem}

\begin{proof} Theorem \ref{theorem more properties} gives $A^{\core_m,W} W (AW)^{k+1}=(AW)^k$. For $X$ of the form (\ref{jed6-W-mWC}), we observe that
$$XW(AW)^{k+1}=BA^{\core_m,W}W(AW)^{k+1}+Y(I_n-WAW A^{\core_m,W})W(AW)^{k+1}=B(AW)^k,$$
that is, $X$ is a solution to (\ref{jed5-W-mWC}).

If (\ref{jed5-W-mWC}) has a solution $X$, then
\begin{eqnarray*}
BA^{\core_m,W}&=&B(AW)^kA[(WA)^{k+m+2}]^\dag (WA)^{2m}[(WA)^m]^\dag\\
&=&XW(AW)^{k+1}A[(WA)^{k+m+2}]^\dag (WA)^{2m}[(WA)^m]^\dag\\
&=&X(WA)^{k+2}[(WA)^{k+m+2}]^\dag (WA)^{2m}[(WA)^m]^\dag\\
&=&XWAWA^{\core_m,W}.
\end{eqnarray*}
Thus, $X=BA^{\core_m,W}+X(I_n-WAW A^{\core_m,W})$ has a form (\ref{jed6-W-mWC}).
\end{proof}

Theorem \ref{te12-W-mWC} gives the next result for the $m$-weak core inverse.

\begin{corollary} Let $A \in \Cnn$, $\ind(A)=k$, $m\in \mathbb{N}$, and $B\in\Cnn$. The general solution to the equation
\begin{equation}XA^{k+1}=BA^k\label{jed5'-W-mWC}\end{equation} is given as
$$X=BA^{\core_m}+Y(I_n-AA^{\core_m}),$$ for arbitrary $Y\in\Cnn$.
\end{corollary}

In order to confirm Theorem \ref{te12-W-mWC}, we give the following example.

\begin{example} Let $A$ and $W$ be given as in Example \ref{ex1-W-mWC},
$$B=\left[\begin{array}{cccc}
2 &  1 & 0 & 0\\
1 &  1 & 0 & 0\\
0 &  0 & 0 & 0\\
0 &  0 & 0 & 0
\end{array}\right] \quad \text{and}\quad 
Y=\left[\begin{array}{ccc}
y_1 & y_2 & y_3 \\
u_1 & u_2 & u_3 \\
v_1 & v_2 & v_3 \\
s_1 & s_2 & s_3 
\end{array}\right],$$ for arbitrary $y_1, y_2, y_3, u_1, u_2, u_3, v_1, v_2, v_3, s_1, s_2, s_3\in\C$.
Since
\begin{eqnarray*}
X&=&BA^{\core_m,W}+Y(I_n-WAW A^{\core_m,W})=\left[\begin{array}{ccc}
2 & y_2 & 0 \\
1 & u_2 & 0 \\
0 & v_2 & 0 \\
0 & s_2 & 0
\end{array}\right],
\end{eqnarray*}
it follows
$$XW(AW)^{k+1}=\left[\begin{array}{cccc}
2 &  0 & 4 & 4\\
1 &  0 & 2 & 2\\
0 &  0 & 0 & 0\\
0 &  0 & 0 & 0
\end{array}\right]=B(AW)^k.$$
\end{example}

\section*{Declarations}

\noindent {\bf  Funding}

\noindent This work was supported by Universidad Nacional de R\'{\i}o Cuarto (Grant PPI 18/C559),  CONICET (Grant PIBAA 28720210100658CO), and by Universidad Nacional de La Pampa, Facultad de Ingenier\'ia (Grant Resol. Nro. 135/19).
The second author is supported by the Ministry of Education, Science and Technological Development,
Republic of Serbia, grant number 451-03-65/2024-03/200124.

\noindent {\bf Authors' contributions}

\noindent The two authors contributed equally to this work.

\end{document}